\documentclass[12pt,oneside,reqno]{amsart}
\usepackage{amsmath,amsthm,amssymb,epic,eepic}
\usepackage{eqlist,eqparbox}
\usepackage{diagrams}
\usepackage{cancel}
\usepackage{comment}

\newcommand{\strih}[2]{{\ulcorner{#1}\urcorner}_{#2}}
\newcommand{\dstrih}[2]{{\llcorner{#1}\lrcorner}_{#2}}
\newcommand{\gstrih}[2]{g_{{\ulcorner{#1}\urcorner}_{#2}}}
\newcommand{\rimply}{\Rightarrow}

\newcommand{\implik}{\Longrightarrow}

\ifx\pdfoutput\@undefined\usepackage[usenames,dvips]{color}
\else\usepackage[usenames,dvipsnames]{color}
\IfFileExists{pdfcolmk.sty}{\usepackage{pdfcolmk}}{} 
\fi

\definecolor{ChadDarkBlue}{rgb}{.1,0,.2}  
\definecolor{ChadBlue}{rgb}{.1,.1,.5}  
\definecolor{ChadRoyal}{rgb}{.2,.2,.8}  
\definecolor{ChadGreen}{rgb}{0,.4,0}    
\definecolor{ChadRed}{rgb}{.5,0,.5}  
\def\zmena#1{{#1}} 
\def\zruseno#1{} 

\title {\bf On tense MV-algebras}
\author {Michal Botur, Jan Paseka}
\thanks{Both authors gratefully acknowledge  the support by ESF Project CZ.1.07/2.3.00/20.0051
Algebraic methods in Quantum Logic of the Masaryk University. 
M. Botur gratefully acknowledges Financial Support 
of the  the Grant Agency of the Czech
Re\-pub\-lic un\-der the grant
No.~   GA\v CR P201/11/P346.}
\address{Palack\' y University Olomouc, Faculty of Sciences, t\v r. 17.listopadu 1192/12, Olomouc 771 46, Czech Republic}
\email{michal.botur@upol.cz}
\address{Department of Mathematics and Statistics, Faculty of Science,
Masaryk University, {Kotl\'a\v r{}sk\' a\ 2}, 611~37 Brno, 
Czech Republic}

\email{paseka@math.muni.cz}

\keywords{Tense operators, MV-algebra, predicate logic}
\subjclass{06D35, 06F35, 03G10}
\begin{document}

\begin{abstract}
The main aim of this article is to study tense MV-algebras which are just MV-algebras with 
new unary operations $G$ and $H$ which express a universal time quantifiers. 
Tense MV-algebras were introduced by D.~Diagonescu and G.~Georgescu. 
Using a new notion of an fm-function between MV-algebras 
we \zruseno{will prove} \zmena{settle a half of their Open problem  about representation 
for some classes of tense MV-algebras,  i.e., we show}  
that any tense semisimple MV-algebra is induced by a time frame 
analogously to classical works in this field of logic. As a by-product we obtain a new characterization  
of extremal states on MV-algebras.
\end{abstract}

\maketitle

\newtheorem{definition}{Definition}
\newtheorem{proposition}{Proposition}
\newtheorem{theorem}{Theorem}
\newtheorem{lemma}{Lemma}
\newtheorem{claim}{Claim}
\newtheorem{cor}{Corollary}
\newtheorem{corollary}{Corollary}
\newtheorem*{example}{Example}
\newtheorem{remark}{Remark}
\newtheorem{open}{Open problem}

\section{Introduction}

Propositional logic usually do not incorporate the dimension of time. To obtain  the \zmena{so-called} 
{\it tense logic}, the propositional calculus is enriched by adding new unary operators $G$ and $H$ 
(and new derived operators $F:=\neg G\neg$ and $P:=\neg H \neg,$ where $\neg$ 
\zmena{denotes} the classical negation connective) which are called {\it tense operators}. The operator $G$ usually expresses the quantifier `it will still be the case that' and $H$ expresses `it has always been the case that'. Hence, $F$ and $P$ are in fact tense existential quantifiers.

If $T$ is non-empty set and $\rho$ a binary relation on $T,$ the couple $(T,\rho)$  is called a {\it time frame}. 
For a given logical formula $\phi$ of our propositional logic and for $t\in T$ we say that $G(\phi)(t)$ is valid if $\phi(s)$ 
is valid for any $s\in T$ with $t\rho s.$ Analogously, $H(\phi)(t)$ is valid if $\phi(s)$ is valid for any $s\in T$ with $s\rho t.$ 
Thus $F(\phi)(t)$ is valid if there exists $s\in T$ with $t\rho s$ and $\phi (s)$ is valid and 
analogously $P(\phi)(t)$ is valid if there exists $s\in T$ with $s \rho t$ and $\phi(s)$ is valid in the propositional logic.

Study of tense operators was originated in 1980's, see e.g. a compendium \cite{2}. Recall that for a classical propositional calculus represented by means of a Boolean algebra $\mathbf B=(B,\vee,\wedge,\neg,0,1),$ tense operators were axiomatized in \cite{2} by the following axioms:
\begin{itemize}
\item[(B1)] $G(1)=1,$ $H(1)=1,$
\item[(B2)] $G(x\wedge y)= G(x)\wedge G(y),$ $H(x\wedge y)=H(x)\wedge H(y),$
\item[(B3)] $\neg G\neg H (x)\leq x,$ $\neg H\neg G (x)\leq x.$
\end{itemize}
For Boolean algebras, the axiom (B3) is equivalent to
\begin{itemize}
\item[(B3')] $G(x)\vee y=x\vee H(y).$
\end{itemize}

To introduce tense operators in non-classical logics, some more axioms must be added  on $G$ and $H$ to express connections with additional operations or logical connectives. For example, for intuitionistic logic (corresponding to Heyting algebras) it was done in \cite{3}, for algebras of logic of quantum mechanics see \cite{4} and \cite{5}, for so called basic algebras it was done in \cite{1}, for other interesting algebras the reader is referred to \cite{8}, \cite{9} and \cite{10}.

Among algebras connected with many-valued logic, let us mention MV-algebras and \L uka\-siewicz-Moisil algebras. Tense operators for the previous cases were introduced and studied in \cite{6} and \cite{7}. Contrary to Boolean algebras where the representation problem \zmena{through a time frame} 
is solved completely, authors in \cite{7} only mention that this problem for MV-algebras was not treated. 
Hence, our \zmena{main goal is to 
find a suitable time frame for given tense operators on a semisimple MV-algebra, i.e., to 
solve the representation problem for semisimple MV-algebras.}

This problem was solved by the first author\footnote{It is not published.} for such tense MV-algebras that the 
tense operators $G$ and $H$ preserve all powers of the operations $\oplus$ and $\odot$. The second author generalized his results replacing original term $t_q$ (see \cite{paseka})  constructed for any rational $q$ by the 
Teheux's term (see Section 2.2 of this paper or \cite{paseka}). Here, we present more general concept 
of used ideas for obtaining stronger results. The main representation theorem for semisimple tense MV-algebras and 
second author's results are corollaries of this.

The paper is divided as follows. In Section \ref{Preliminaries} we recall the basic fundamental results on MV-algebras and tense MV-algebras, and in this way fix the notation and terminology.  
Afterwards we summarize some folklore 
results on MV-terms $t_r(x)$ produced only from operations of the form 
$x\oplus x$ and $x\odot x$. Then in Section \ref{semistate} we introduce a 
notion of a semi-state on an MV-algebra and we show that any  semi-state is a meet of 
extremal states. Also, we give a new characterization of extremal states on MV-algebras.
In Section \ref{construction} we 
introduce the notions of an {fm-function between MV-algebras} 
({strong fm-function between MV-algebras}). We establish  
a canonical construction of strong fm-function between MV-algebras. 

In Section \ref{repres} we solve the representation problem for  fm-function 
between semisimple MV-algebras. Moreover, we show that in this case 
they coincide with strong  fm-functions. 
Finally we prove the representation theorem for semisimple tense MV-algebras.

\section{Preliminaries}\label{Preliminaries}
\subsection{MV-algebras and tense operators}
The concept of MV-algebras was introduced by C.C. Chang in \cite{Chan} as algebraic counterpart 
of the \L uka\-siewicz multi valued logic (see \cite{Luk}). 
Recall, that by an {\em MV-algebra} is meant an algebra $\mathbf A=(A,\oplus,\neg,0)$ of type $(2,1,0)$ satisfying the axioms:
\begin{itemize}
\item[(MV1)] $x\oplus y = y\oplus x,$
\item[(MV2)] $x\oplus (y\oplus z) =(x\oplus y)\oplus z,$
\item[(MV3)] $x\oplus 0=x,$
\item[(MV4)] $\neg\neg x=x,$
\item[(MV5)] $x\oplus 1=1,$ where $1:=\neg 0,$
\item[(MV6)] $\neg(\neg x\oplus y)\oplus y = \neg (\neg y\oplus x)\oplus x.$
\end{itemize} 

The order relation $\leq$ can be introduced on any MV-algebra $\mathbf A$ by the stipulation $$x\leq y\mbox{ if and only if } \neg x\oplus y =1.$$ Moreover, the ordered set $(A,\leq)$ can be organized into a bounded lattice $(A,\vee,\wedge,0,1)$ where $$x\vee y = \neg (\neg x\oplus y)\oplus y \mbox{ and } x\wedge y = \neg(\neg x\vee \neg y).$$

Besides of these, we can introduce two more interesting operation $\odot$ and $\rightarrow$ by setting $$x\odot y :=\neg(\neg x\oplus\neg y) \mbox{ and }x\rightarrow y := \neg x\oplus y.$$ 
Those operations are connected by the adjointness property
$$x\odot y\leq z\mbox{ if and only if } x\leq y\rightarrow z.$$

An MV-algebra is said to be {\em linearly ordered} 
(or an {\em MV-chain}) if its order is linear.

Given a positive integer $n\in {\mathbb N}$, 
we let $n\times x= x \oplus x \oplus x \cdots
  \oplus x$, $n$ times, $x^{n}= x \odot x \odot x \cdots
  \odot x$, $n$ times, $0x = 0$ and $x^{0} = 1$.

In every MV-algebra the following equalities hold: 
\begin{enumerate}
\item[{\rm (D1)}] $a\oplus  \bigvee_{i\in I} x_i=%
\bigvee_{i\in I} (a\oplus x_i)$, 
$a\oplus  \bigwedge_{i\in I} x_i=%
\bigwedge_{i\in I} (a\oplus x_i)$, 
\item[{\rm (D2)}] $a\odot  \bigvee_{i\in I} x_i=%
\bigvee_{i\in I} (a\odot x_i)$, 
$a\odot  \bigwedge_{i\in I} x_i=%
\bigwedge_{i\in I} (a\odot x_i)$,
\end{enumerate}

\noindent{}whenever the respective sides are defined.

An element $a$ of an MV-algebra $\mathbf A$ is said to be {\it Boolean} if
$a\oplus a = a$. We say that an MV-algebra $\mathbf A$ is {\it Boolean} if
every element of $\mathbf A$ is Boolean. For an MV-algebra $\mathbf A$, 
the set $B({\mathbf A})$ of all Boolean elements is a Boolean algebra. 

{\em Morphisms of MV-algebras} (shortly {\em MV-morphisms}) are defined as usual, \zmena{i.e.,} 
they are functions which preserve the binary operations $\oplus$ and $\odot$, the unary operation 
$\neg$ and the constants $0$ and $1$. 

Tense MV-algebras were introduced by D. Diagonescu and G.Georgescu in \cite{7}. 

\begin{definition}\label{tense}{\rm 
Let us have an MV-algebra $\mathbf A=(A,\oplus,\neg,0)$. We say that $(\mathbf A,G,H)$ is a 
{\em tense MV-algebra} and $G$ and $H$ are {\em tense operators} 
if $G$ and $H$ are a unary operations on $A$ satisfying:
\begin{itemize}
\item[(i)] $G(1)=H(1)=1,$
\item[(ii)] $G(x)\odot G(y)\leq G (x\odot y),$ $H(x)\odot H(y)\leq H(x\odot y),$
\item[(iii)] $G(x)\oplus G(y)\leq G (x\oplus y),$ $H(x)\oplus H(y)\leq H(x\oplus y),$
\item[(iv)] $G(x)\odot G(x)=G(x\odot x),$ $H(x)\odot H(x)=H(x\odot x)$,
\item[(v)] $G(x)\oplus G(x)=G(x\oplus x),$ $H(x)\oplus H(x)=H(x\oplus x)$,
\item[(vi)] $\neg G\neg H (x)\leq x,$ $\neg H\neg G (x)\leq x.$ 
\end{itemize}}
\end{definition}


Applying the axioms (i) and (ii), we get immediately monotonicity of the operators $G$ and $H.$ Thus, if $x\leq y$ for any $x,y\in A$ then $G(x)\leq G(y)$ and $H(x)\leq H(y).$ 

We note that the original definition of tense MV-algebras \cite[Proposition 5.1, Remark 5.1]{7} use alternative inequalities
\begin{itemize}
\item[(ii')] $G(x\rightarrow y)\leq G(x)\rightarrow G(y),$ $H(x\rightarrow y)\leq H(x)\rightarrow H(y),$
\item[(vi')] $x\leq G\neg H\neg x,$ $x\leq H\neg G\neg x.$
\end{itemize}

Monotonicity of the operators $G$ and $H$ and adjointness property give equivalence of (ii) and (ii'). Using double negation law and antitonicity of the negation we obtain equivalence of (vi) and (vi').

The following theorem \zmena{describes} the most important construction of tense MV-algebras.

\begin{theorem}\cite{7}
Let $\mathbf A$ be a linearly ordered complete MV-algebra and let $T$ be any set with a binary relation $\rho\subseteq T^2.$ Then $(\mathbf A^T,G^*,H^*)$ where operations $G^*$ and $H^*$ are calculated point-wise
$$G^*(x)(i):=\bigwedge_{i\rho j} x(j)\quad \mbox{ and }\quad H^*(x)(i):= \bigwedge_{j\rho i} x(j)$$
is a tense MV-algebra. In this case we say that the tense MV-algebra $(\mathbf A^T,G^*,H^*)$ is induced by the frame $(T,\rho).$
\end{theorem}

We will prove that any couple of tense operators on any semisimple MV-algebra can be embedded into $([0,1]^T,G^*, H^*)$ where $G^*$ and $H^*$ are  \zmena{tense} operators induced by some time frame 
$(T,\rho)$. \zmena{For related results on more general  operators on any semisimple MV-algebra 
see  \cite{paseka}.} 

\subsection{Dyadic numbers and MV-terms}
\label{Dyadic}

The contents of this part summarizes the basic results about some folklore 
results of some MV-terms from \cite{teheux} and \cite{pascal}. The techniques described here  
have been used already in \cite{BK}
and later, e.g., in \cite{Mun}, \cite{DN}, \cite{hansoul} and \cite{teheux2}.


We remark some concepts introduced by B. Teheux in \cite{teheux}.   The set $\mathbb D$ of dyadic numbers is the set of the rational numbers that can be written as a finite sum of power of 2.
If $a$ is a number of  $[0, 1]$, a dyadic decomposition of $a$ is a sequence 
$a^{*} = (a_i )_{ i \in {\mathbb N}}$ of elements of $\{ 0, 1 \}$ such that 
$a = \sum^{\infty}_{i=1} a_i 2^{-i}$. We denote by $a^{*}_ i$ the $i^{\text{th}}$ element of any sequence (of length greater than $i$) $a^{*}$.
If $a$ is a dyadic number of $[0, 1]$, then $a$ admits a unique finite dyadic decomposition, 
called the {\em dyadic decomposition of} $a$.
If $a^{*}$  is a dyadic decomposition of a real $a$ and if $k$ is a positive integer 
then we denote by $\strih{a^{*}}{k}$ the finite sequence 
$(a_1, \dots, a_k)$ defined by the first $k$ elements of $a^{*}$ and by 
$\dstrih{a^{*}}{k}$ the dyadic number 
$\sum^{k}_{i=1} a_i 2^{-i}$.
We denote by $f_0(x)$ and $f_1(x)$ the terms $x\oplus x$ and 
$x\odot x$ respectively, and by
$T_{\mathbb D}$  the clone generated by $f_0(x)$ and $f_1(x)$. 

We also denote by $g_{.}$ the mapping between the set of finite sequences 
of elements of $\{0, 1\}$ (and thus of dyadic numbers in $[0, 1]$) 
and $T_{\mathbb D}$ defined by:
$$g_{(a_1,...,a_k)} = f_{a_k} \circ \dots \circ f_{a_1}$$
for any finite sequence $(a_1, . . . , a_k)$ of 
elements of $\{0, 1\}$.  If 
$a = \sum^{k}_{i=1} a_i 2^{-i}$, we sometimes
write $g_a$ instead of $g_{(a_1,...,a_k)}$. 

\zruseno{We also denote,
for a dyadic number $a\in {\mathbb D}\cap [0,1]$ and a positive integer 
$k\in {\mathbb N}$ such that $2^{-k}\leq 1-a$, by 
$l(a,k):[a, a+2^{-k}]\to [0,1]$ a linear function defined as follows 
$l(a,k)(x)=2^{k}(x-a)$ for all $x\in [a, a+2^{-k}]$.}

\begin{lemma}{\rm \cite[Lemma 1.14]{teheux}}  If $a^{*} = (a_i)_{i \in\mathbb N}$  
and $x^{*} = (x_i )_{i \in\mathbb N}$ are dyadic decompositions of two elements
of $a, x\in [0, 1]$, then, for any positive integer ${k\in\mathbb N}$,
$$\gstrih{a^{*}}{k}
(x) = \left\{%
\begin{array}{l l}
1& \text{if}\ x > \sum^{k}_{i=1} a_i 2^{- i} + 2^{-k}\\
0& \text{if}\ x < \sum^{k}_{i=1} a_i 2^{- i} 
\end{array}
\right.
$$
\end{lemma}

Note that for any finite sequence $(a_1,...,a_k)$ of elements 
of $\{0, 1\}$ such that 
$a_k=0$ we have that $g_{(a_1,...,a_k)}= g_{(a_1,...,a_{k-1})}%
\oplus g_{(a_1,...,a_{k-1})}$ and clearly any dyadic number $a$ corresponds 
to such a sequence $(a_1,...,a_k)$. 

As an immediate consequence, we get

\begin{cor}{\rm \cite[Corollary 1.15 (1)]{teheux}} \label{odhad} Let us have the standard MV-algebra 
$[0,1]$, $x\in [0, 1]$ and $r\in (0,1)\cap {\mathbb D}$. Then 
 there is a term $t_r$ in $T_{\mathbb D}$ such that
$$
t_r(x)=1 \quad \text{if and only if}\quad r\leq x.
$$
\end{cor}

\subsection{Filters, ultrafilters and the term $t_r$}
\label{improve}

The aim of this part is to show that any filter $F$ in an 
MV-algebra  ${\mathbf A}$ which does not contain the element $t_r(x)$  for some dyadic number 
$r\in (0,1)\cap {\mathbb D}$ and an element $x\in A$ can be extended to an ultrafilter $U$ containing 
$F$ such that $t_r(x)\not\in U$.

\bigskip

A \emph{filter} of an MV-algebra ${\mathbf A}$ is a subset $F \subseteq A$
satisfying: 

(F1) $1 \in F$

(F2) $x \in F,\; y \in A,\; x \leq y\; \rimply \; y \in F$

(F3)  $x, y \in F \;  \rimply \; x\odot y \in F$.

\bigskip

A  filter is said to be  \emph{proper} if $0 \notin F$. Note that there 
is a one-to-one correspondence between filters and congruences 
on MV-algebras. A filter $Q$ is \emph{prime} if it satisfies the following
conditions:

(P1) $0 \notin Q$.

(P2) For each $x, \, y $ in $A$ such that $x\vee y\in Q$, 
     either $x \in Q$ or $y\in Q$.

\bigskip
In this case the corresponding factor MV-algebra ${\mathbf A}/Q$  is 
linear. 

\bigskip

A filter $U$ is {\em maximal} (and in this case it will 
be also called an ultrafilter) if $0\notin U$ and 
for any other filter $F$ of ${\mathbf A}$ such that $U\subseteq F$, then either $F=A$ or $F=U$. There is a one-to-one correspondence between 
ultrafilters and MV-morphisms from ${\mathbf A}$ into $[0,1]$ (extremal states). \zmena{ For any 
ultrafilter $A\in T$ we identify the class $x/A$ with its image in the standard algebra 
and thus with its image in interval $[0,1]$ of real numbers.}

In what follows we work mostly with MV-morphisms into $[0,1]$ instead of ultrafilters. 

\begin{lemma}\label{oT4} Let $\mathbf{A}$ be a linearly ordered MV-algebra, 
$s:\mathbf A\to [0,1]$ an MV-morphism, $x\in A$ such that
$s(x)=1$. Then  $x\oplus x=1$.
\end{lemma}
\begin{proof} Assume that $x \leq \neg x$. Then 
$1=s(x)\odot s(x)=s(x\odot x)\leq s(x\odot \neg x)=s(0)=0$ which is absurd. 
Therefore $\neg x < x$ and we have that 
$ x\oplus x\geq x\oplus \neg x=1$.
\end{proof}

\newcommand{\indentitem}{\setlength\itemindent{-25pt}} 
\begin{proposition}\label{T4} Let $\mathbf{A}$ be a linearly ordered MV-algebra, 
$s:\mathbf A\to [0,1]$ an MV-morphism, $x\in A$.  Then 
$s(x)=1$ iff $t_r(x)=1$ for all $r\in (0,1)\cap {\mathbb D}$.

Equivalently, $s(x)<1$ iff there is a dyadic number $r\in (0,1)\cap {\mathbb D}$ 
such that $t_r(x)\not=1$. In this case, $s(x)<r$. 
\end{proposition}
\begin{proof} In what follows we may assume that $x\not=0$ since 
$s(0)=0$ and $t_r(0)=0$ for all $r\in (0,1)\cap {\mathbb D}$. Note first that $s(t_r(x))=t_r(s(x))$ since $s$ is 
an MV-morphism. Then $s(x)=1$ iff $r\leq s(x)$ for all $r\in (0,1)\cap {\mathbb D}$ iff $t_r(s(x))=1$  for all $r\in (0,1)\cap {\mathbb D}$  iff 
$s(t_r(x))=1$ for all $r\in (0,1)\cap {\mathbb D}$. 

Assume now that $t_r(x)=1$ for all $r\in (0,1)\cap {\mathbb D}$. Then 
evidently $s(t_r(x))=1$ for all $r\in (0,1)\cap {\mathbb D}$ and by the above considerations we have that $s(x)=1$. 

Conversely, let $s(x)=1$ and $r\in (0,1)\cap {\mathbb D}$. Then 
$t_r(x)= t(x)\oplus t(x)$ such that $t(x)$ 
is some term from the clone $T_{\mathbb D}$ constructed entirely from the operations $(-)\oplus  (-)$ 
and $(-)\odot (-)$. Therefore $s(t(x))=t(s(x))=t(1)=1$. By Lemma \ref{oT4} we get that 
$t_r(x)=t(x)\oplus t(x)=1$.

\end{proof}

\begin{proposition}\label{cT4} Let $\mathbf{A}$ be an MV-algebra, 
$x\in A$ and $F$ be any filter of $\mathbf{A}$. Then there 
is  an MV-morphism $s:\mathbf A\to [0,1]$ such that 
$s(F)\subseteq \{1\}$ and $s(x)<1$ if and only if 
there is a dyadic number $r\in (0,1)\cap {\mathbb D}$ such that $t_r(x)\notin F$.
\end{proposition}
\begin{proof} 
Assume first that  there is an MV-morphism $s:\mathbf A\to [0,1]$ such that 
$s(F)\subseteq \{1\}$ and $s(x)<1$. Then there is a dyadic number 
$r\in (0,1)\cap {\mathbb D}$ such that $s(x)<r<1$. By Corollary \ref{odhad} 
we get that $s(t_r(x))=t_r(s(x))\not =1$. Hence 
$t_r(x) \notin F$. 

Now, let there be a dyadic number $r\in (0,1)\cap {\mathbb D}$ such that 
$t_r(x) \notin F$.  Then there is a filter $K$ of $\mathbf{A}$, 
$F\subseteq K$, $t_r(x)\notin K$ such that $K$ is maximal with this property. 
Evidently, $K$ is a prime filter of $\mathbf{A}$. Hence the factor algebra 
$\mathbf{A}/K$ is linearly ordered and we have a surjective MV-morphism 
$g:\mathbf{A}\to \mathbf{A}/K$, $g(K)\subseteq \{1\}$. 
Let us denote by $U_K$ the maximal 
filter of $\mathbf{A}/K$ and by $s_K:\mathbf{A}/K \to [0,1]$ the 
corresponding MV-morphism. Because $t_r(x) \notin K$ we get that 
$t_r(g(x))=g(t_r(x))\not =1$.

It follows from Proposition \ref{T4} that $s_K(t_r(g(x)))< r <1$. 
This yields that $s_K(g(t_r(x)))< r <1$. Let us put $s=s_K\circ g$. Then 
$s:\mathbf A\to [0,1]$ is an MV-morphism, $s(t_r(x))< r <1$. 
Evidently $s(x)< r < 1$ otherwise we would have also $1=s(t_r(x))<1$, 
a contradiction. Clearly, 
$s(F)\subseteq s(K)=s_K(g(K))\subseteq s_K(\{1\})=\{1\}$. 
\end{proof}

\begin{cor}\label{oddels}
Let $\mathbf{A}$ be an MV-algebra, 
$x\in A$ and $F$ be any filter of  $\mathbf A$ such that $t_r(x)\notin F$
for some dyadic number $r\in (0,1)\cap {\mathbb D}$. Then there 
is an ultrafilter $U$ of $\mathbf A$ such that 
$F\subseteq U$ and $x/U< r<1$.
\end{cor}

\section{Semistates on MV-algebras} \label{semistate}

In this section we characterize arbitrary meets of MV-morphism into 
a unit interval as so-called semi-states.

\begin{definition}\label{semist} Let $\mathbf{A}$ be an MV-algebra. A map 
$s:\mathbf{A} \to [0,1]$ is called
\begin{enumerate}
\item {\em a semi-state on}  $\mathbf{A}$ if 
\begin{itemize}
\item[(i)] $s(1)=1,$
\item[(ii)]$x\leq y$ implies $s(x)\leq s(y),$
\item[(iii)] $s(x)=1$ and $s(y)=1$ implies $s (x\odot y)=1,$ 
 \item[(iv)] $s(x)\odot s(x)=s(x\odot x),$ 
\item[(v)] $s(x)\oplus s(x)=s(x\oplus x).$
\end{itemize}
\item{\em a  strong semi-state on}  $\mathbf{A}$ if it is a semistate such that 
\begin{itemize}
\item[(vi)] $s(x)\odot s(y)\leq s (x\odot y),$
\item[(vii)] $s(x)\oplus s(y)\leq s (x\oplus y),$
\item[(viii)] $s(x\wedge y)=s(x)\wedge s(y),$
 \item[(ix)] $s(x^{n})=s(x)^{n}$ for all $n\in{\mathbb N}$,
\item[(x)] $n\times s(x)= s(n\times x)$ for all $n\in{\mathbb N}$,    
\end{itemize}

\end{enumerate}

\end{definition}

Note that any MV-morphism into a unit interval is a  strong semi-state.

\begin{lemma}\label{meetstate} Let $\mathbf{A}$ be an MV-algebra, 
$S$ a non-empty set of semi-states (strong semi-states) on $\mathbf{A}$. Then the point-wise meet 
$t=\bigwedge S:\mathbf{A} \to [0,1]$ is a semi-state 
(strong semi-state) on  $\mathbf{A}$.
\end{lemma}
\begin{proof} Let us check the conditions (i)-(vi) from Definition \ref{semist}.

(i): Clearly, $t(1)=\bigwedge \{s(1) \mid s\in S\} =\bigwedge \{1 \mid s\in S\}=1$. 

(ii): Assume $x\leq y$. Then  $t(x)=\bigwedge \{s(x) \mid s\in S\} %
\leq\bigwedge \{s(y) \mid s\in S\}=t(y)$. 

(iii): Let $t(x)=\bigwedge \{s(x) \mid s\in S\}=1$ and 
$t(y)=\bigwedge \{s(y) \mid s\in S\}=1$.  
It follows, that, for all $s\in S$, $s(x)=1=s(y)$. Hence also 
$s(x\odot y)=1$. This yields that $t(x\odot y)=\bigwedge \{s(x\odot y) \mid s\in S\}=1$.

(iv, v): Since $[0,1]$ is linearly ordered we have (by taking in the respective 
part of the proof either the minimum of $s_1(x)$ and $s_2(x)$ or 
the  maximum of $s_1(x)$ and $s_2(x)$)
$$
\begin{array}{r c l}
t(x)\odot t(x)&=&\bigwedge \{s_1(x) \mid s_1\in S\} \odot %
\bigwedge \{ s_2(x) \mid s_2\in S\}=%
\bigwedge \{s_1(x) \odot s_2(x) \mid s_1, s_2\in S\} \\[0.1cm]
&\geq&\bigwedge \{s(x) \odot s(x) \mid s\in S\}=%
\bigwedge \{s(x\odot x) \mid s\in S\}= t (x\odot x),
\end{array} 
$$
$$
\begin{array}{r c l}
t(x)\odot t(x)&=&\bigwedge \{s_1(x) \mid s_1\in S\} \odot %
\bigwedge \{ s_2(x) \mid s_2\in S\}=%
\bigwedge \{s_1(x) \odot s_2(x) \mid s_1, s_2\in S\} \\[0.1cm]
&\leq&\bigwedge \{s(x) \odot s(x) \mid s\in S\}=%
\bigwedge \{s(x\odot x) \mid s\in S\}= t (x\odot x),
\end{array} 
$$
$$
\begin{array}{r c l}
t(x)\oplus t(x)&=&\bigwedge \{s_1(x) \mid s_1\in S\} \oplus %
\bigwedge \{ s_2(x) \mid s_2\in S\}=%
\bigwedge \{s_1(x) \oplus s_2(x) \mid s_1, s_2\in S\} \\[0.1cm]
&\geq&\bigwedge \{s(x) \oplus s(x) \mid s\in S\}=%
\bigwedge \{s(x\oplus x) \mid s\in S\}= t (x\oplus x),
\end{array} 
$$

and 
$$
\begin{array}{r c l}
t(x)\oplus t(x)&=&\bigwedge \{s_1(x) \mid s_1\in S\} \oplus %
\bigwedge \{ s_2(x) \mid s_2\in S\}=%
\bigwedge \{s_1(x) \oplus s_2(x) \mid s_1, s_2\in S\} \\[0.1cm]
&\leq&\bigwedge \{s(x) \oplus s(x) \mid s\in S\}=%
\bigwedge \{s(x\oplus x) \mid s\in S\}= t (x\oplus x).
\end{array} 
$$

(vi): Let us compute the following
$$
\begin{array}{r c l}
t(x)\odot t(y)&=&\bigwedge \{s_1(x) \mid s_1\in S\} \odot %
\bigwedge \{ s_2(y) \mid s_2\in S\}=%
\bigwedge \{s_1(x) \odot s_2(y) \mid s_1, s_2\in S\} \\[0.1cm]
&\leq&\bigwedge \{s(x) \odot s(y) \mid s\in S\}\leq %
\bigwedge \{s(x\odot y) \mid s\in S\}= t (x\odot y).
\end{array} 
$$

(vii): Applying the same considerations as in (vi) we have
$$
\begin{array}{r c l}
t(x)\oplus t(y)&=&\bigwedge \{s_1(x) \mid s_1\in S\} \oplus %
\bigwedge \{ s_2(y) \mid s_2\in S\}=%
\bigwedge \{s_1(x) \oplus s_2(y) \mid s_1, s_2\in S\} \\[0.1cm]
&\leq&\bigwedge \{s(x) \oplus s(y) \mid s\in S\}\leq %
\bigwedge \{s(x\oplus y) \mid s\in S\}= t (x\oplus y).
\end{array} 
$$

(viii): Similarly,  $t(x\wedge y)=\bigwedge \{s(x\wedge y) \mid s\in S\} %
=\bigwedge \{s(x)\wedge s(y) \mid s\in S\}=%
\bigwedge \{s(x) \mid s\in S\} \wedge \bigwedge \{ s(y) \mid s\in S\}=%
t(x)\wedge t(y)$. 

(ix): Assume that $x\in A$ and ${n\in \mathbb N}$. We have, 
repeatedly using  the equality (D2) from the Introduction, that 
  
  $$\begin{array}{r c l}
   (\bigwedge\{s(x); s\in S\})^{n} &
  = &\bigwedge\{s_1(x)\odot s_2(x) \odot \dots \odot 
   s_n(x)\mid s_1, \dots s_n\in S\}\\[0.1cm] 
   &\leq&  \bigwedge\{s(x)\odot s(x) \odot \dots \odot 
   s(x)\mid s_1, \dots s_n\in S, \\
   & &\phantom{\bigwedge\{ }%
    s\in \{s_1, \dots, s_n\}, s(x) = 
   \max \{s_1(x), \dots s_n(x)\} \}\\
   &=&\bigwedge\{s(x)^n\mid s\in S\}=\bigwedge\{s(x^n)\mid s\in S\}
    \end{array}
  $$
  and similarly 
$$\begin{array}{r c l}
   (\bigwedge\{s(x); s\in S\})^{n} &
  = &\bigwedge\{s_1(x)\odot s_2(x) \odot \dots \odot 
   s_n(x)\mid s_1, \dots s_n\in S\}\\[0.1cm] 
   &\geq&  \bigwedge\{s(x)\odot s(x) \odot \dots \odot 
   s(x)\mid s_1, \dots s_n\in S, \\
   & &\phantom{\bigwedge\{ }%
    s\in \{s_1, \dots, s_n\}, s(x) = 
   \min \{s_1(x), \dots s_n(x)\} \}\\
   &=&\bigwedge\{s(x)^n\mid s\in S\}=\bigwedge\{s(x^n)\mid s\in S\}
    \end{array}
  $$
  Thus $(\bigwedge\{s(x); s\in S\})^{n}=\bigwedge\{s(x^n)\mid s\in S\}$.

(x): It follows by the same considerations as for (ix) applied to $\oplus$ and repeatedly using  the equality (D1) from the Introduction. 
\end{proof}

\begin{lemma}\label{porovb}  Let $\mathbf{A}$ be an MV-algebra, $s, t$ 
semi-states on  $\mathbf{A}$. Then $t\leq s$ iff 
$t(x)=1$ implies $s(x)=1$ for all $x\in A$.
\end{lemma}
\begin{proof} Clearly,  $t\leq s$ yields the condition
$t(x)=1$ implies $s(x)=1$ for all $x\in A$.

Assume now that $t(x)=1$ implies $s(x)=1$ for all $x\in A$ is valid and 
that there is $y\in A$ such that $s(y) < t(y)$. Thus, there is a 
dyadic number $r\in (0,1)\cap\mathbb D$ such that $s(y)<r< t(y).$ 
By Corollary \ref{odhad} there is a term 
 $t_r$ in $T_{\mathbb D}$ such that
$t_r(s(y))< 1$ and $t_r(t(y))=1$. It follows that 
$s(t_r(y))=t_r(s(y))< 1$ and $t(t_r(y))=t_r(t(y))=1$. The last condition yields that 
$s(t_r(y))=1$, a contradiction.
\end{proof}

\begin{proposition}\label{prusekapr}  Let $\mathbf{A}$ be an MV-algebra, $t$ 
a semi-state on  $\mathbf{A}$ and 
$S_{t}=\{ s: \mathbf{A}\to [0,1]\mid \ s\ \text{is an}\ \text{MV-morphism},  s\geq t\}$. 
Then $t=\bigwedge S_t$.
\end{proposition}
\begin{proof} Clearly,  $t\leq\bigwedge S_t$. Assume that there is 
$x\in A$ such that  $t(x)<\bigwedge S_t (x)$. Thus, there is a 
dyadic number $r\in (0,1)\cap\mathbb D$ such that $t(x)<r<\bigwedge S_t (x).$ 
Again by  Corollary \ref{odhad} there is a term 
 $t_r$ in $T_{\mathbb D}$ such that 
$t(t_r(x))=t_r(t(x))<1$. Let us put 
$F=\{z\in A\mid t(z)=1\}$.  
The set $F$ is by the condition (iii)  a filter of $\mathbf{A}$, 
$t_r(x)\not\in F$. Hence there is by Proposition \ref{oddels} an MV-morphism  
 $s:\mathbf{A}\to [0,1]$ such that 
$s(F)\subseteq \{1\}$ and $s(x)< r<1$. It follows by Lemma \ref{porovb} 
that  $t\leq s$, i.e., $s\in S_t$ and  $s(x)< r <\bigwedge S_t (x) \leq  s(x)$, 
a contradiction.
\end{proof}

\begin{corollary}\label{semiismeet} Any semi-state on  an MV-algebra  $\mathbf{A}$ is 
a  strong semi-state.
\end{corollary}

\begin{corollary}\label{nenul} The only  semi-state $s$ on  an MV-algebra  $\mathbf{A}$ 
with $s(0)\not=0$ is the constant function $s(x)=1$ for all $x\in A$.
\end{corollary}

\begin{corollary}\label{consts} The only  semi-state $s$ on  the standard 
MV-algebra  $[0, 1]$  with $s(0)=0$ is the identity function.
\end{corollary}

\begin{remark}\label{dualsemi} {\rm It is transparent that all the preceding notions and results including 
Proposition  \ref{oddels} can be dualized. 
In particular, any dual semi-state, i.e., a map $s:\mathbf{A} \to [0,1]$ satisfying 
conditions (i),(ii),(iv), (v) and the dual condition 
(iii)' $s(x)=0$ and $s(y)=0$ implies $s (x\oplus y)=0$ is a join of extremal states on 
$\mathbf{A}$.}
\end{remark}

\begin{proposition}\label{eqtkapr}  Let $\mathbf{A}$ be an MV-algebra, $s$ 
a state on  $\mathbf{A}$. 
Then the following conditions are equivalent:
\begin{itemize}
\item[(a)] s is a morphism of MV-algebras,
\item[(b)] s satisfies the condition $s(x\wedge x')=s(x)\wedge s(x)'$ for all $x\in A$.,
\item[(c)] s satisfies the condition (iv) from Definition \ref{semist},
\item[(d)] s satisfies the condition (viii) from Definition \ref{semist}
\end{itemize}
\end{proposition}
\begin{proof} (a) $\implik$ (b): It is evident. 

(b) $\implik$ (c):

(c) $\implik$ (d): Clearly, any state satisfies conditions (i) and (ii) from  Definition \ref{semist}. 
Let us check the condition (iii). Assume that $s(x)=1=s(y)$. Then 
$s(x\odot y)=s(x'\oplus  y')'=s(x'+ x\wedge y')'\geq (s(x')+ s(x\wedge y'))'=(0+0)'=1$.

By the assumption (b) we have that  the condition (iv) is satisfied and for any state the condition (v) is 
equivalent with (iv). It follows that $s$  is a semi-state. By Corollary \ref{semiismeet} 
$s$ is a  strong semi-state, i.e., (viii) is satisfied.

(d) $\implik$ (a): It follows from \cite[Lemma 3.1]{pulm}.
\end{proof}


\section{Functions between MV-algebras and  their construction}
\label{construction}

This section studies the notion of an {fm-function between MV-algebras} 
({strong fm-function between MV-algebras}). The main purpose of this section 
is to establish in some sense 
a canonical construction of strong fm-function between MV-algebras. This 
construction is an ultimate source of numerous examples. 

\begin{definition}
  By an \textbf{fm-function between MV-algebras} $G$ is meant a 
function  $G:\mathbf{A}_1\to \mathbf{A}_2$ such that 
 $\mathbf{A}_1=(A_1;\oplus_1,\odot_1, \neg_1, 0_1,1_1)$ and 
$\mathbf{A}_2=({A}_2;\oplus_2,\odot_2, \neg_2, 0_2,1_2)$
  are MV-algebras and 
  \begin{itemize}
  \settowidth{\leftmargin}{(iiiii)}
\settowidth{\labelwidth}{(iii)}
\settowidth{\itemindent}{(ii)}
    \item[(FM1)] $G(1_1)=1_2$,
    \item[(FM2)] $x\leq_1 y$ implies $G(x)\leq_2 G(y)$,
    \item[(FM3)] $G(x)=1_2=G(y)$ implies $G(x\odot_1 y)=1_2$,
    \item[(FM4)] $G(x)\odot_2 G(x)=G(x\odot_1 x)$,
    \item[(FM5)] $G(x)\oplus_2 G(x)= G(x\oplus_1 x)$.
  \end{itemize}
  
 If moreover $G$ satisfies conditions 
 \begin{itemize}    
\settowidth{\leftmargin}{(iiiii)}
\settowidth{\labelwidth}{(iii)}
\settowidth{\itemindent}{(ii)}
    \item[(FM6)] $G(x)\odot_2 G(y)\leq G(x\odot_1 y)$,
    \item[(FM7)] $G(x)\oplus_2 G(y)\leq G(x\oplus_1 y)$,
    \item[(FM8)] $G(x)\wedge_2 G(y)= G(x\wedge_1 y)$,
    \item[(FM9)] $G(x^{n})=G(x)^{n}$ for all $n\in{\mathbb N}$,
    \item[(FM10)] $n\times_2 G(x)= G(n\times_1 x)$ for all $n\in{\mathbb N}$,

  \end{itemize}
  we say that $G$ is a
\textbf{strong fm-function between MV-algebras}.

If $G:\mathbf{A}_1\to \mathbf{A}_2$ and $H:\mathbf{B}_1\to \mathbf{B}_2$ are 
fm-functions between MV-algebras, then a {\em morphism between $G$ and $H$} 
 is a pair $(\varphi,\psi)$ of morphism of MV-algebras $\varphi:\mathbf{A}_1 \to \mathbf{B}_1$ 
and $\psi:\mathbf{A}_2 \to \mathbf{B}_2$ such that 
$\psi(G(x))=H(\varphi(x))$, for any $x\in A_1$.
\end{definition}

Note that (FM8) yields (FM2), (FM9) yields (FM4) and (FM10) yields (FM5). Also, 
a composition of  fm-functions (strong fm-functions)  is an fm-function (a strong fm-function) again 
and any morphism of MV-algebras is an  fm-function (a strong fm-function).

The notion of  an fm-function generalizes both the notions of a semi-state and of 
a $\odot$-operator from \cite{paseka} which is 
an fm-function $G$ from  $\mathbf{A}_1$ to itself such that (FM6) is satisfied.

According to both (FM4) and (FM5), $G|_{B(\mathbf{A}_1)}: {B(\mathbf{A}_1)} \to {B(\mathbf{A}_2)}$ 
is an fm-function (a strong fm-function) whenever 
$G$ has the respective property.

\begin{lemma} Let $G:\mathbf{A}_1\to \mathbf{A}_2$  be an fm-function between MV-algebras, 
 $r\in (0,1)\cap {\mathbb D}$. Then 
$t_r(G(x))= G(t_r(x))$ for all $x\in A_1$. 
\end{lemma}
\begin{proof} Note that $G(x)\oplus_2 G(x)= G(x\oplus_1 x)$  by (FM5) and 
$G(x)\odot_2 G(x)= G(x\odot_1 x)$ by (FM4). 
Then, since $t_r\in T_{\mathbb D}$ is defined inductively using only 
the operations $(-)\oplus  (-)$ 
and $(-)\odot (-)$, we get  $t_r(G(x))=G(t_r(x))$. 
\end{proof}


By a \textbf{frame}  is meant a triple $(S,T,R)$ where $S,T$ are non-void sets and $R\subseteq S\times T$.  Having an MV-algebra $\mathbf{M}=(M;\oplus,\odot, \neg, 0,1)$ and a non-void set $T$, we can produce the direct power
$\mathbf{M}^T=(M^T;\oplus, \odot, \neg,o,j)$ where the operations 
$\oplus$, $\odot$ and $\neg$ are defined and evaluated on $p,q\in M^T$ componentwise. Moreover, $o, j$ are such elements of $M^T$ that $o(t)=0$ and $j(t)=1$
for all $t\in T$. The direct power $\mathbf{M}^T$ is again an MV-algebra. 

The notion of frame allows us to construct new examples of MV-algebras with a strong operator.

\begin{theorem}\label{nvc}
 Let $\mathbf{M}$ be a linearly ordered complete MV-algebra, 
 $(S, T, R)$ be a frame 
 and $G^{*}$ be a map from 
 ${M}^T$ into ${M}^S$ defined by 
$$G^{*}(p)(s) = \bigwedge\{p(t) \mid t \in T, sRt\},$$ 

\noindent{}for all $p\in M^T$  and $s\in S$.  Then 
$G^*$ is a strong fm-function between MV-algebras which 
has a left adjoint $P^{*}$. In this case, 
for all $q\in M^S$  and $t\in T$,  
$$P^{*}(q)(t) = \bigvee\{q(s) \mid s \in T, sRt\}$$ 
\noindent{}and $P^*:(\mathbf{M}^S)^{op}\to (\mathbf{M}^T)^{op}$ 
is a strong fm-function between MV-algebras. 
\end{theorem}
\begin{proof} %
The conditions (FM1)-(FM10)  can be easily shown by the same considerations as in 
\cite[Theorem 3.4]{paseka} and/or Lemma \ref{meetstate}.

  Moreover, for any $p\in M^{T}$ and $q\in M^{S}$,  we can compute: 
    $$
	\begin{array}{@{}r@{\,}c@{\,}l}q(s)\leq G^*(p)(s)\ \text{for all}\ 
     s\in T&\Longleftrightarrow %
    &q(s)\leq \bigwedge\{p(t) | t \in T, sRt\}%
    \ \text{for all}\ 
     s\in T\\[0.1cm] 
    &\Longleftrightarrow& %
     q(s)\leq p(t)\ \text{for all}\  s, t \in T, sRt\\[0.1cm]
    &\Longleftrightarrow&%
   \bigvee_{i\in I}\{q(s) | s \in T, sRt\}\leq p(t)%
   \ \text{for all}\ 
     t\in T\\[0.1cm]
     &\Longleftrightarrow&%
    P^*(q)(t)\leq p(t)%
   \ \text{for all}\ t\in T.
  \end{array}$$
  
  This yields that $q\leq G^*(p)$ iff $P^*(q)\leq p$. Then 
   $P^{*}$ is a left adjoint of $G^{*}$. Hence $G^{*}$ preserves 
arbitrary meets.


\end{proof}

We say that $G^*:\mathbf{M}^T \to \mathbf{M}^S$ is the {\em canonical strong fm-function 
between MV-algebras induced  by the frame} $(S,T,R)$  
{\em and the MV-algebra $\mathbf{M}$}.

\section{The main theorem and its applications}\label{repres}

Before proving our main theorem, we remark that {\it semisimple} MV-algebras \zmena{\cite{Mun}} are just 
subdirect products of the simple MV-algebras. Any simple MV-algebra is uniquelly embeddable into  the standard MV-algebra on the interval $[0,1]$ of reals. It is known that an MV-algebra is semisimple if and only if the intersection of the set of its maximal (prime) filters is equal to the set $\{1\}.$ Note also 
that any  complete MV-algebra is semisimple.

Hence a semisimple MV-algebra $\mathbf A$ is embedded into $[0,1]^T$  (see \cite{belluce})
where $T$ is the set of all \zmena{ultrafilters} (morphisms into the standard MV-algebra )
and $\pi_{F}(x)=x(F)=x/F\in [0,1]$ for any $x\in \mathbf S\subseteq [0,1]^T$ and 
any $F\in T$; here $\pi_{F}:[0,1]^T\to [0,1]$ is the respective projection onto $[0,1]$.
 \zruseno{Thus for any maximal ultrafilter $A\in T$ we identify
the class $x/F$ with its image in the standard MV-algebra and thus with its image in interval $[0,1]$ of real numbers.}

\begin{theorem}\label{mainth}
Let  $G:\mathbf A_1\to \mathbf A_2$  be an fm-function 
between semisimple MV-algebras, $T$ a set of all MV-morphism from $\mathbf A_1$ 
to the standard MV-algebra $[0,1]$ and $S$ a set of all MV-morphism 
from $\mathbf A_2$ to $[0,1]$. 

Further, let $(S, T, \rho_G)$ be a frame such that 
the relation $\rho_G\subseteq S\times T$ is defined by 
$$s\rho_G t\mbox{ if and only if } s(G(x))\leq t(x)\mbox{ for any } x\in A_1.$$

Then $G$ is representable via the  canonical strong fm-function
$G^*:[0,1]^T \to [0,1]^S$  between MV-algebras induced  by the frame $(S,T,\rho_G)$  
and the standard MV-algebra $[0,1]$, i.e., the following diagram of fm-functions commutes:

$$
\begin{diagram}
\mathbf A_1&%
\rTo(2,0)^{\quad\quad{}G}&&\mathbf A_2&\qquad&\qquad&\\
\dTo(0,3)^{i_{\mathbf A_1}^T}&&&\dTo(0,3)_{{i_{\mathbf A_2}^S}}&\qquad&\qquad&\\
[0,1]^T&\rTo(2,0)_{\quad\quad\quad{}G^*}&&[0,1]^S&.\qquad&\qquad&
\end{diagram}
$$

\end{theorem}
\begin{proof} Assume that $x\in A_1$ and $s\in S$. Then 
${i_{\mathbf A_2}^S}(G(x))(s)=s(G(x))\leq t(x)$ for all $t\in T$, $(s,t)\in \rho_G$. It 
follows that ${i_{\mathbf A_2}^S}(G(x))\leq G^{*}({i_{\mathbf A_1}^T}(x))$. 

Note that $s\circ G$ is a semi-state on $\mathbf A_1$ and by Proposition 
\ref{prusekapr} we get that 
$$
\begin{array}{r c l}
s\circ G&=&\bigwedge \{ t: \mathbf{A}_1\to [0,1]\mid \ t\ \text{is an}\ \text{MV-morphism},  
t\geq s\circ G\}\\[0.1cm]
&=&\bigwedge \{ t\in T\mid \ (s, t)\in  \rho_G\}.
\end{array}$$
This yields that actually ${i_{\mathbf A_2}^S}(G(x))= G^{*}({i_{\mathbf A_1}^T}(x))$. 
\end{proof}

\begin{proposition}\label{duals} For any MV-algebra $\mathbf A_1$, 
any semisimple MV-algebra $\mathbf A_2$  with  a set $S$ of all MV-morphism 
from $\mathbf A_2$ to $[0,1]$ and any map 
$G:A_1\to  A_2$ the following conditions are 
equivalent:
\begin{enumerate}
\item[{\rm(i)}] $G$ is an fm-function between MV-algebras.
\item[{\rm(ii)}] $G$ is  a strong fm-function between MV-algebras.
\end{enumerate}
\end{proposition}
\begin{proof} (i) $\Longrightarrow$ (ii): Note that the composition
$\pi_{s}\circ {i_{\mathbf A_2}^S}\circ G$ is 
 a strong semi-state for any $s\in T$.  It follows that 
${i_{\mathbf A_2}^S}\circ G$ is 
 a strong fm-function between MV-algebras.
Since the embedding ${i_{\mathbf A_2}^S}:\mathbf A_2 \to [0,1]^{S}$ 
reflects order we obtain that conditions (FM6)-(FM10) are satisfied.

\noindent{}(ii) $\Longrightarrow$ (i): It is evident.
\end{proof}

\begin{open} Find MV-algebras $\mathbf A_1$ and $\mathbf A_2$ with 
an  fm-function $G$ between them  such that  $G$ is not  a strong fm-function. 
\end{open}

Note that our approach of using semi-states in the above proof of 
Theorem \ref{mainth} also covers the main result of the paper \cite{paseka} which 
is Theorem 4.5 from \cite{paseka}.

\begin{theorem}  \cite[Theorem 4.5]{paseka}\label{strongreprest}
{\bf (Representation theorem for MV-algebras with an $\odot$-operator)}
For any semisimple MV-algebra $\mathbf A$ with an $\odot$-operator $G$,  
$\mathbf A$ is embeddable via 
MV-operator morphism $i_{\mathbf A}^{T}$ into the 
canonical MV-algebra ${\mathcal L}_G=([0,1]^{T};G^{*})$ with 
a strong operator $G^{*}$ induced by the 
canonical frame  $(T,\rho_G)$ and the standard MV-algebra $[0,1]$.
Further, for all $x\in M$ and for all $s\in T$, 
$s(G(x))=G^{*}((t(x))_{t\in T})(s)$. 
\end{theorem}

The next Theorem, which is a solution of 
a half of the Open problem 5.1 in \cite{7}, was first proved by the first author of the paper 
for tense operators satisfying (FM9) and (FM10) and then, based on the idea of  the first author, 
proved in the full generality by the second author as 
Theorem 5.5 from \cite{paseka}, can be written as follows.

\begin{theorem}\label{maintense}
Let  $\mathbf A$  be a semisimple MV-algebra with tense operators $G$ and $H$. Then $(\mathbf S,G,H)$ can be embedded into the tense MV-algebra $([0,1]^T,G^*,H^*)$ induced by the frame $(T,\rho_G),$ where $T$ is the set of all maximal proper filters and the relation $\rho_G$ is defined by 
$$A\rho_G B\mbox{ if and only if } G(x)/A\leq x/B\mbox{ for any } x\in S.$$
\end{theorem}
\begin{proof}
First, let us  define a second relation $\rho_H\subseteq T^2$ by the stipulation:
$$B\rho_H A\mbox{ if and only if } H(x)/B\leq x/A \mbox{ for any } x\in S.$$
\begin{claim}
The equality $\rho_G=\rho_H^{-1}$ holds.
\end{claim}
\begin{proof}
Let us suppose that $A\rho_G B$ for some $A,B\in T.$ Due to Definition \ref{tense} vi) we have $\neg G\neg H(x)\leq x$ and $\neg x/A\leq G\neg H(x)/A.$ $A\rho_G B$ yields $G\neg H(x)/A\leq \neg H(x)/B$ and together $\neg x/A\leq \neg H(x)/B$ yields $H(x)/B\leq x/A$ for any $x\in S.$

Due to the definition of $\rho_H$ we have $B\rho_H A$ and $\rho_G\subseteq \rho_H^{-1}.$  Analogously we can prove the second inclusion.

\end{proof}

The remaining part follows from Theorem \ref{mainth}. 
Basically, the obtained equations $G^*(x) = G(x)$ and $H^*(x)=H(x)$ finish the proof.
\end{proof}

 \begin{theorem}
 a) If $([0,1]^T,G^*,H^*)$ is a tense MV-algebra induced by a time frame $(T,\rho),$ then
 \begin{itemize}
 \item[(i)] if $\rho$ is reflexive then $G^*(x)\leq x$ and $H^*(x)\leq x$ hold for any $x\in [0,1]^T,$
 \item[(ii)] if $\rho $ is symmetric then $G^*(x)=H^*(x)$ holds for any $x\in [0,1]^T$,
\item[(iii)] if $\rho$ is transitive then $G^*G^*(x)\geq G^*(x)$ and $H^*H^*(x)\geq H^*(x)$ hold for any $x\in [0,1]^T$.
 \end{itemize} 

b) Let $(\mathbf S,G,H)$ be a semisimple tense MV-algebra and $(T,\rho_G)$  the
time frame   which induces the tense MV-algebra  $([0,1]^T, G^*, H^*)$ by Theorem \ref{maintense}. Then 
\begin{itemize}
 \item[(i)] if $G(x)\leq x$ and $H(x)\leq x$ hold for any $x\in S$ then $\rho_G$ is reflexive,
 \item[(ii)] if $G(x)=H(x)$ holds for any $x\in S$ then $\rho_G $ is symmetric,
\item[(iii)] if $GG(x)\geq G(x)$ and $HH(x)\geq H(x)$ hold for any $x\in S$ then $\rho_G$ is transitive.
 \end{itemize}
 \end{theorem}
\begin{proof}
ai) If the relation $\rho$ is reflexive, then $i\rho i$ yields $G^*(x)(i)=\bigwedge_{i\rho j}x(j)\leq x(i)$ for any $i\in T.$ 
The part for $H^*$ we can prove analogously.

aii) If $\rho$ is symmetric then $G^*(x)(i)=\bigwedge_{i\rho j} x(j)=\bigwedge_{j\rho i} x(j)=H^*(x)(i)$ 
for any $i\in T$ which clearly yields $G^*=H^*.$

aiii) If $\rho $ is transitive then $\{x(k)\, |\, i\rho j\mbox{ and } j\rho k\}\subseteq\{x(k)\, |\, i\rho k\}$ 
and then
\begin{eqnarray*}
G^*G^*(x)(i)&=&\bigwedge_{i\rho j}G^*(x)(j)=\bigwedge_{i\rho j}\bigwedge_{j\rho k}x(k)\\
&=&\bigwedge\{x(k)\, |\, i\rho j\mbox{ and } j\rho k\}
\geq \bigwedge_{i\rho k} x(k)= G^*x(i)
\end{eqnarray*}
holds for any $i\in T.$

b) We remark that relation $\rho_G$ in Theorem \ref{maintense} is defined by
$$A\rho_G B\mbox{ if and only if } G(x)/A\leq x/B\mbox{ for any }x\in A.$$

bi) If $G(x)\leq x$ for any $x\in A$ then $G(x)/A\leq x/A$ holds for any $x\in A$ and thus $A\rho_G A.$ Together $\rho_G$ is reflexive.

bii) The Claim 1 in the proof of Theorem \ref{maintense} shows that $G(x)/A\leq x/B$ for any $x\in A$ if and only if $H(x)/B\leq x/A$ for any $x\in A.$ If $G=H$ holds then $G(x)/A\leq x/B$ for any $x\in A$ if and only if $G(x)/B\leq x/A$ for any $x\in A$ and consequently $A\rho_G B$ holds 
if and only if $B\rho_G A$ holds. Thus the relation $\rho_G$ is symmetric.

biii) Let us suppose that $G(x)\leq GG(x)$ for any $x\in A.$ If $A\rho_G B$ and $B\rho_G C$ 
hold then any $x\in A$ satisfies $G(x)/A\leq GG(x)/A\leq G(x)/B\leq x/C$ which yields $A\rho_G C.$ Thus the relation $\rho_G$ is transitive. 
\end{proof}

\begin{remark} {\rm Note that one can extend the number of fm-functions between 
MV-algebras arbitrarily and our results remain valid.  Similarly as for semi-states in Remark 
\ref{dualsemi} we could 
introduce the notion of a dual (strong) fm-function and}  all the preceding results would also remain valid 
in this dual setting.
\end{remark}

\section{Concluding remarks}
 We have settled a half of the Open problem 5.1 in \cite{7} using a more general approach 
of fm-functions. The remaining part asks about the existence 
of a representation theorem for any tense MV-algebra via Di Nola representation theorem for MV-algebras. We hope that our  results 
will be a next step in obtaining a general representation theorem for  tense MV-algebras. 

We expect that our method can be easily applied to modal or 
similar operators that may be treated as universal quantifiers 
on various types of MV-algebras



\begin{thebibliography}{99}
\bibitem{belluce} {\sc Belluce L.P.}: {\it Semisimple algebras of infinite-valued logic and bold fuzzy set theory}, 
Can. J. Math. {\bf 38} (1986) 1356–-1379.
\bibitem{1}{\sc Botur M., Chajda I., Hala\v s R., Kola\v r\' ik M.}: {\it Tense operators on  Basic Algebras}, 
International Journal of Theoretical Physics \textbf{50}, (2011),  3737--3749.
\bibitem{2}{\sc Burges J.}: {\it Basic tense logic}, in D.M.Gabbay, F.G\" unther (Eds), Handbook of Philosophical Logic, vol. II,  D. Reidel Publ. Comp., 1984, 89-139.
\bibitem{BK}{\sc Butnariu D.,  Klement, E. P.}: 
{\it Triangular Norm-Based Measures and Games with Fuzzy 
Coalitions.} Kluwer Academic Publishers, Dordrecht (1993).
\bibitem{3}{\sc Chajda I.}: {\it Algebraic axiomatization of tense intuitionistic logic}, Cent. Europ. J. Math., {\bf 9} (2011).
\bibitem{4}{\sc Chajda I., Kola\v r\' ik M.}: {\it Dynamic effect algebras}, Mathematica Slovaca \textbf{62} (2012), 379--388.
\bibitem{5}{\sc Chajda I., Paseka J.}: {\it Dynamic effect algebras and their representation}, 
Soft. Comput.  \textbf{16} (2012), 1733--1741.
\bibitem{Chan}{\sc Chang C.C.}: {\it Algebraic analysis of many valued logics.} Trans. Amer. Math. Soc. {\bf 88} (1958), 467–490.
\bibitem{6}{\sc Chiri\c t\u a C.}: {\it Tense $\theta$-valued Moisil propositional logic}, Int. J. of Computers, Communications and Control, {\bf 5} (2010), 642-653.
\bibitem{Mun} {\sc Cignoli R. L. O., D'Ottaviano  I. M. L., Mundici  D.}: {\it Algebraic Foundations of Many-valued Reasoning,} Kluwer (2000).
\bibitem{7}{\sc Diaconescu D., Georgescu G.}: {\it Tense operators on MV-algebras and \L ukasiewicz-Moisil algebras}, Fund. Inform. {\bf 81}(2007), 379-408.
\bibitem{DN} {\sc Di Nola, A., Navara, M.}: 
{\it The $\sigma$-complete MV-algebras which have enough states},  
Colloquium Math. {\bf 103} (2005), 121--130.
\bibitem{8}{\sc Figallo A.V., Pelaitay G.}: {\it Tense operators on SHn-algebras}, Pioneer J. of Algebra, Number Theory and Appl. {\bf 1}(2011),33-41.
\bibitem{9}{\sc Figallo A.V., Gallardo G., Pelaitay G.}:{\it Tense operators on m-symetric algebras}, Int. Math. Forum {\bf 41}(2011), 2007-2014.
\bibitem{hansoul} {\sc Hansoul G., Teheux B.}: {\it  Completeness results for many-valued \L{}ukasiewicz modal systems 
	and relational semantics}, 2006. 	Available at http://arxiv.org/abs/math/0612542.
\bibitem{10}{\sc Kowalski T.}: {\it Varieties of tense algebras}, Rep. Math. Logic, {\bf 32}(1998), 53-95.
\bibitem{Luk} {\sc \L ukasiewicz J.}: \textit{On three-valued logic}, in L.~Borkowski (ed.), Selected works by Jan \L ukasiewicz,
North–-Holland, Amsterdam, 1970, pp. 87–-88.
\bibitem{pascal} {\sc Ostermann P.}: {\it Many-valued modal propositional   calculi}, Z.  Math. Logik Grundlag. Math., {\bf 34} (1988) 343--354. 
\bibitem{paseka}{\sc Paseka J.}: {\it Operators on MV-algebras and their representations}, 
Fuzzy Sets and Systems, doi: 10.1016/j.fss.2013.02.010.
\bibitem{pulm}{\sc Pulmannov\' a\ S.}: {\it On fuzzy hidden variables}, 
Fuzzy Sets and Systems, {\bf 155}(2009), 119-137.
\bibitem{teheux2} {\sc Teheux B.}: {\it A Duality for the Algebras
of a \L{}ukasiewicz $n+1$-valued Modal System} Studia Logica {\bf 87} (2007) 13-–36, doi: 10.1007/s11225-007-9074-5.
\bibitem{teheux} {\sc Teheux B.}: {\it Algebraic approach to modal extensions of \L{}ukasiewicz logics}, Doctoral thesis, Universit\'e de Liege, 2009, http://orbi.ulg.ac.be/\mbox{handle/2268/10887}.
\end{thebibliography}
\end{document}